\documentclass[12pt,leqno]{amsart}
\usepackage[a4paper,top=2cm,bottom=2cm,left=2cm,right=2cm]{geometry}
\usepackage{amssymb,amsmath}
\usepackage{subfig}
\usepackage[english]{babel}
\usepackage{graphicx}
\usepackage{tabularx}
\usepackage{enumerate, hyperref}

\def\ZZ{\mathbb Z}

\def\R{\mathbb R}

\def\F{\mathcal F}

\newtheorem{theorem}{Theorem}
\newtheorem{lemma}[theorem]{Lemma}

\newtheorem{proposition}[theorem]{Proposition}
\newtheorem{corollary}[theorem]{Corollary}

\title{Kadec-$1/4$ Theorem for Sinc Bases.}
\author{Antonio Avantaggiati, Paola Loreti, Pierluigi Vellucci}

\begin{document}
\maketitle

\begin{abstract}
In this paper we show two results. 
In the first result we consider
$\lambda_n-n=\frac{A}{n^\alpha}$ for $n\in\mathbb N$; if $\alpha>1/2$ and $0<A<\frac{1}{\pi\sqrt{2 \sqrt{2}\zeta(2\alpha)}}$, the system $\left\{\operatorname{sinc}( \lambda_n - t)\right\}_{n\in\mathbb N}$ is a Riesz basis for $PW_{\pi}$.
With the second result, we study the stability of $\left\{\operatorname{sinc}( \lambda_n - t)\right\}_{n\in\ZZ}$ for $\lambda_n\in\mathbb C$; if $|\lambda_n-n|\leqq L<\frac{1}{\pi}\, \sqrt\frac{3\alpha}{8}$, for all $n\in\mathbb Z$, then $\{\operatorname{sinc}(\lambda_n-t)\}_{n\in\mathbb Z}$ forms a Riesz basis for $PW_{\pi}$. Here $\alpha$ is the Lamb-Oseen constant.
\end{abstract}

\tableofcontents

\section{Introduction}
\label{sec:intro}

Let $f$ be a function which can be expanded as
\begin{equation}
\label{eq:6}
f(t)=\sum_{n\in\mathbb Z} c_n \operatorname{sinc} (t-n)
\end{equation}
where
\begin{equation}
\label{eq:sinc}
 {\operatorname{sinc}}(\alpha)=
   \begin{cases}
   \frac{ \sin(\pi \alpha)}{\pi \alpha}  \qquad &\alpha \not= 0,\\
   1\qquad & \alpha=0,
    \end{cases}
\end{equation}
is the \emph{normalized sinc function}. The RHS of (\ref{eq:6}) is called ``cardinal series'' or \emph{Whittaker cardinal series}. A major factor affecting current interest in the cardinal series is its importance for certain applications as, for example, interpolation based on (\ref{eq:6}) which is usually called \emph{ideal bandlimited interpolation} (or sinc interpolation), because it provides a perfect reconstruction for all t, if $f(t)$ is bandlimited in $[-\pi,\pi]$ and if the sampling frequency is greater that the so-called \emph{Nyquist rate}. The system used to implement (\ref{eq:6}) is also known in in engineering applications as  \emph{ideal DAC} (i.e. digital-to-analog converter, see \cite{Man11}). The presence of the perturbation could lose the correct reconstruction of the function (signal), so it is important to study the conditions for which the system is still able to reconstruct the function (signal) belonging to a given space.
Other applications are sampling theory of band-limited signals in communications engineering \cite{Hig85} or sinc-quadrature method for differential equations \cite{Lu92}. The so-called \emph{sinc numerical methods of computation}, provide procedures for function approximation over bounded or unbounded regions, encompassing interpolation, approximation of derivatives, approximate definite and indefinite integration, and so on \cite{Ste00}. These problems motivated our investigation on sinc systems.
\begin{figure}[tb]
\centering
\includegraphics[scale=0.40]{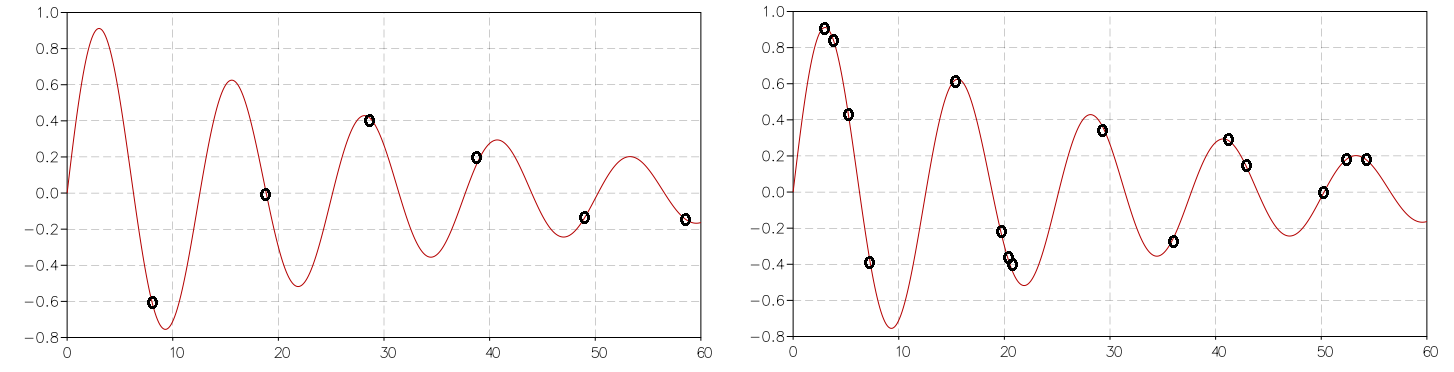}
\caption{Left: A function $f$ defined on $\mathbb R$ has been sampled on a uniformly spaced set. Right: The same function f has been sampled on a non-uniformly spaced set.}
\label{fig:5}
\end{figure}
\begin{figure}[tb]
\centering
\includegraphics[scale=0.4]{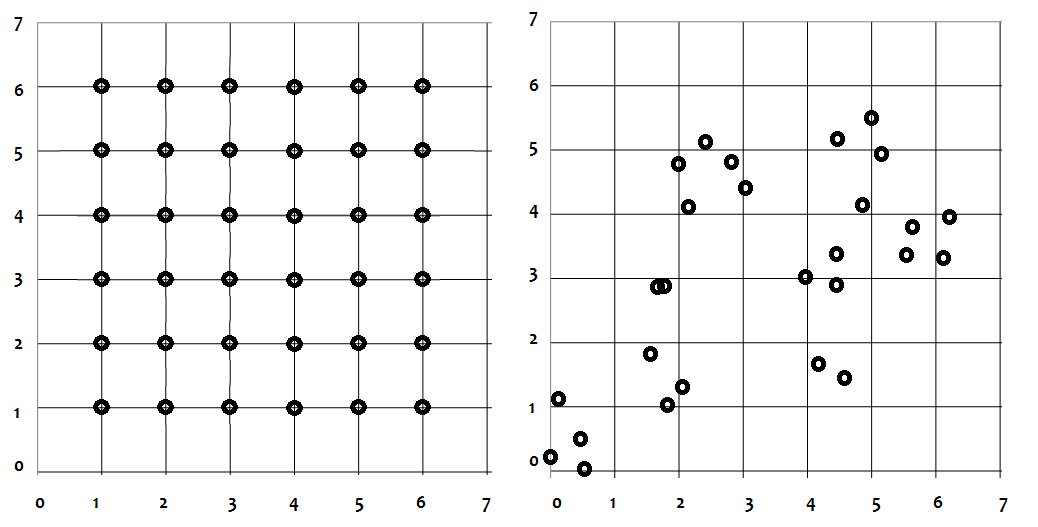}
\caption{Sampling grids. Left: uniform cartesian sampling. Right: A typical nonuniform sampling set as encountered in various signal and image processing applications.}
\label{fig:6}
\end{figure}

For these reasons, the cardinal series have been widely discussed in the literature; see also \cite{Whi15} and \cite{Whi64}. They are linked to a classical basis, the exponentials $\{e^{in t}\}_{n\in \mathbb Z}$ in $L^2(-\pi,\pi)$, through the \emph{Fourier transform}, indeed formally
$$\F\left(e^{i t\mu}\chi_{[-\pi,\pi]}\left(t\right)\right)(\xi)=\int_{-\pi}^\pi e^{i(\mu-\xi)t} dt=2\pi\operatorname{sinc}(\mu-\xi).$$
Studies on more general exponential systems $\{e^{i\lambda_n t}\}_{n\in \mathbb Z}$ find their origin in the celebrated 1934's work of Paley and Wiener \cite{Pal} in $L^2(0, T)$, where $T > 0$. They proved that if $\lambda_n\in \mathbb R, n\in \mathbb Z$ and
$$|\lambda_n-n|\leq L< \pi^{-2} \ \ \ \ n\in \mathbb Z$$
then the system $\{e^{i\lambda_n t} \}_{n\in \mathbb Z}$ forms a Riesz basis in $L^2[-\pi,\pi]$. A well-known theorem by Kadec \cite{Kad}, \cite{You} shows that $1/4$ is a stability bound for the exponential basis on $L^2(-\pi,\pi)$, in the sense that for $L<1/4$, $\{e^{i\lambda_n t} \}_{n\in \mathbb Z}$ is still a Riesz basis in $L^2(-\pi,\pi)$. More than 60 years after Paley and Wiener initiated the study of nonharmonic Fourier series in $L^2[-\pi,\pi]$, many other approaches to exponential Riesz basis problem have emerged in the literature. For other contributions to exponential Riesz basis problem and Kadec's theorem see survey papers, as: \cite{Sed}, \cite{Ul}.

For the system of cardinal sines $\{\operatorname{sinc}(t-n) \}_{n\in \mathbb Z}$ we tried to follow the same approach. For simplicity, we refer to $\{\operatorname{sinc}(t-n) \}_{n\in \mathbb Z}$ with terms of ``sinc system'' or \emph{sinc basis}.

The results of this paper are the following theorems.

First of all, we recall and prove the classical result (see also \cite{Ava}). Below, we denote with $PW_{\pi}$ the \emph{Paley-Wiener space}.
\begin{proposition}
\label{t:alv}
Let $\{\lambda_n\}_{n\in\ZZ}$ be a sequence of real numbers for which
\begin{equation}
|\lambda_n-n|\leqq L<\infty, \ \ n=0, \pm 1, \pm 2, ...
\end{equation}
If $L<\frac 14$, the sequence $\left\{\operatorname{sinc}( \lambda_n - t)\right\}_{n\in\mathbb Z}$ satisfies the Paley-Wiener criterion and so forms a Riesz basis for $PW_{\pi}$. Moreover, constant $1/4$ is optimal.
\end{proposition}

The subsequent results have been achieved in an attempt to reobtain the optimal constant $1/4$ without going to the exponential basis, i.e. working directly on cardinal series. Let us consider the following two results.
\begin{theorem}
\label{t:alv11}
Let $\lambda_n-n=\frac{A}{n^\alpha}$ for $n=1,2,3,\dots$. If $\alpha>1/2$ and $0<A<\frac{1}{\pi\sqrt{2 \sqrt{2}\zeta(2\alpha)}}$ then the system $\left\{\operatorname{sinc}( \lambda_n - t)\right\}_{n\in\mathbb N}$ satisfies the Paley-Wiener criterion and so forms a Riesz basis for $PW_{\pi}$.
\end{theorem}
Numerical evaluation in the case $\lambda_n-n=\frac{A}{n^\alpha}$ is given in Section \ref{sec:ex}; in the Tables are showed that, when $\alpha=1$, for increasing value of $A$ until $A\simeq 0.443\dots$, the system $\left\{\operatorname{sinc}( \lambda_n - t)\right\}_{n\in\mathbb N}$ is a Riesz basis in $PW_{\pi}$. If $n=1,2,\dots$, we have that $\lambda_n-n\leq L$ where $L$ is greater ($\simeq0.443\dots$) of Kadec's bound. This is due to the assumption $\lambda_n-n=\frac{A}{n^\alpha}$, that is, to have considered a non-uniform stability bound.

In Section \ref{sec:complex} we study the stability of $\left\{\operatorname{sinc}( \lambda_n - t)\right\}_{n\in\ZZ}$ for $\lambda_n\in\mathbb C$, reobtaining a stability bound which depends from Lamb-Oseen constant \cite{Oseen}. This constant was also appeared in previous work \cite{Ava} although the stability bound was not correct.

In a previous work one of the author studied the extension to complex numbers of Kadec type estimate for exponential bases \cite{Vel}. The method used there is inspired to work by Duffin and Eachus \cite{2:2}. In \cite{Ava} we performed a preliminary study by adacting a previous result on sinc. Here we give a complete result by the following theorem.
\begin{theorem}
\label{t:1}
If $\{\lambda_n\}$ is a sequence of complex numbers for which
\begin{equation}\label{Kadstyleeq}
|\lambda_n-n|\leqq L<\frac{1}{\pi}\, \sqrt\frac{3\alpha}{8}, \ \ n=0, \pm 1, \pm 2, ...
\end{equation}
then $\{\operatorname{sinc}(\lambda_n-t)\}_{n\in\mathbb Z}$ satisfies the Paley-Wiener criterion and so forms a Riesz basis for $PW_{\pi}$.
\end{theorem}
Observe that the optimality of the bound for the complex case is not studied in our result.

\section{Preliminaries}
\label{sec:pre}
In this section we will introduce some useful notations and results about cardinal series, with reference to applications in sampling. Furthermore, we give a small overview on the Lamb-Oseen constant \cite{Oseen} which is involved in the estimation for the complex case.

\subsection{Sampling Theorem and Stability}
By $L^2({-\infty},{+\infty})$ we denote the Hilbert space of real functions that are square integrable in Lebesgue's sense:
$$L^2(\mathbb R)=\left\{f: \int_{-\infty}^{+\infty} |f(t)|^2 dt<+\infty\right\}$$
with respect to the inner product and $L^2$-norm that, on $\mathbb R$, are
$$\langle f, g \rangle=\frac{1}{2\pi}\int_{-\infty}^\infty f(t) \overline{g(t)} dt\qquad ||f||=\sqrt{\langle f, f \rangle}$$
Given $f\in L^2(\mathbb R)$ we denote by $\F({f})$ the Fourier transform of $f$,
$$\F\left(f\right)(\xi)=\int_{-\infty}^{+\infty} f(t) e^{-i\xi t}dt. $$
Let $e_n$ be an orthonormal basis of an Hilbert space $H$. Then Parseval's identity asserts that for every $x \in H$,
$$\sum_n |\langle x, e_n\rangle|^2 = \|x\|^2.$$
Plancherel identity is expressed, in its common form:
$$\int_{-\infty}^\infty f(t)\, \overline{g(t)}\, dt = \frac{1}{2\pi} \int_{-\infty}^\infty \F\left(f\right)(\xi)\, \overline{\F\left(g\right)(\xi)}\, d\xi.$$
A function $f \in L^2(\mathbb R)$ is band-limited if the Fourier transform $\F\left(f\right)$ has compact support. The Paley-Wiener space $PW_\pi$ is the subspace of $L^2(\mathbb R)$ defined by
$$PW_\pi:=\left\{ f\in L^2(\mathbb R) \Bigl| \ \operatorname{supp}\F\left(f\right)\subseteq [-\pi,\pi]\right\}.$$
We will now recall that the Paley-Wiener space has an orthonormal basis consisting of translates of sinc-function.
\begin{theorem}(Shannon's sampling theorem) \cite{Chr10}, \cite{Je77}, \cite{Ko33}, \cite{Sha}.
\label{th:3}
The functions $\{\operatorname{sinc}(\cdot-n) \}_{n\in \mathbb Z}$ form an orthonormal basis for $PW_\pi$. If $f \in PW_\pi$ is continuous, then
\begin{equation}
\label{eq:8}
f(t)=\sum_{n\in \mathbb Z}f(n)\, \operatorname{sinc}(t-n).
\end{equation}
\end{theorem}
Taking the Fourier transform in equation (\ref{eq:8}) we obtain
\begin{equation}
\F\left(f\right)(\xi)=\sum_{n\in \mathbb Z}\left\langle\F\left(f\right),e^{i n \xi}\right\rangle_{L^2(-\pi,\pi)} e^{i n \xi},
\end{equation}
where $\langle g, h\rangle_{L^2(-\pi,\pi)}=\frac{1}{2\pi}\int_{-\pi}^\pi g(\xi) \overline{h(\xi)} d\xi$.


Usually, it is said that bases in Banach spaces form a \emph{stable} class in the sense that sequences sufficiently close to bases are themselves bases. The fundamental stability criterion, and historically the first, is due to Paley and Wiener \cite{Pal}, \cite{You}.
\begin{theorem}
\label{th:1}
Let $\{x_n\}$ be a basis for a Banach space $X$, and suppose that $\{y_n\}$ is a sequence of elements of $X$ such that
$$\left\|\sum_{i=1}^n a_i (x_i-y_i)\right\|\leq \lambda \left\|\sum_{i=1}^n a_i x_i\right\|$$
for some constant $0\leqq\lambda< 1$, and all choices of the scalars $a_1,\dots,a_n$ ($n=1,2,3,\dots$). Then $\{y_n\}$ is a basis for $X$ equivalent to $\{x_n\}$.
\end{theorem}
We reformulate Theorem \ref{th:1} to be applied to orthonormal bases.
\begin{theorem}
\label{th:2}
Let $\{e_n\}$ be an orthonormal basis for a Hilbert space $H$, and let $\{f_n\}$ be ``close'' to $\{e_n\}$ in the sense that
$$\left\|\sum_{i=1}^n a_i (e_i-f_i)\right\|\leq \lambda \sqrt{\sum |c_i|^2}$$
for some constant $0\leqq\lambda< 1$, and all choices of the scalars $a_1,\dots,a_n$ ($n=1,2,3,\dots$). Then $\{f_n\}$ is a Riesz basis for $H$.
\end{theorem}
\begin{theorem}[Kadec $\frac{1}{4}$-Theorem]
Let $\{\lambda_n\}_{n\in \mathbb Z}$ be a sequence in $\mathbb R$ satisfying
$$|\lambda_n -n|<\frac{1}{4}, \ n=0,\pm 1,\pm2, ...$$
then the set $\{e^{i\lambda_n t} \}_{n\in \mathbb Z}$ is a Riesz basis for $L^2[-\pi,\pi]$.
\end{theorem}
Kadec's theorem has been extensively generalized. See, for example \cite{Avd}, \cite{Bai}, \cite{Khr}, \cite{Pav2}, \cite{Sun}, \cite{Vel}.

\subsection{Lambert function W, Lamb-Oseen constant.}
\label{Lamfunc}

The Lambert function $W$ \cite{Co96}, \cite{Hayes}, \cite{Ste05} is defined by the equation
\begin{equation}
\label{wfunc}
W(x) e^{W(x)}=x
\end{equation}
The function $f(\xi)=\xi e^\xi$ for $\xi\in \mathbb R$ has a strict minimum point in $\xi=-1$.
We draw the picture of $\xi e^\xi=f(\xi)$ in figure (\ref{fig:1}).
\begin{figure}[tb]
\centering
\includegraphics[scale=0.80]{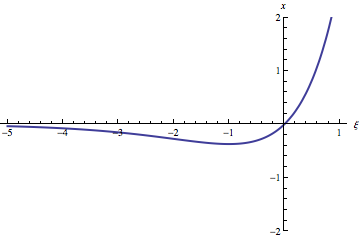}
\caption{Diagram of $\xi e^\xi=f(\xi)$.}
\label{fig:1}
\end{figure}
In \cite{Ava} it is proved the following proposition.
\begin{proposition}
The function $f(\xi)=\xi e^\xi$ has an increasing inverse in $(-1,+\infty)$, and a decreasing inverse in $(-\infty,-1)$.
\end{proposition}
We consider $f(\xi)=\xi e^\xi$ restricted to the interval $(-\infty,-1]$ and we denote by  $W_{-1}$ its  inverse. $W_{-1}$ is defined in the interval $[-1/e,0)$.
We have two identities arising from the definition of  $W_{-1}$:
\begin{equation}
\label{id1}
W_{-1}(\xi e^\xi)=\xi, \left[\Leftrightarrow W_{-1}\left[f(\xi)\right]=W_{-1}\left[\xi e^\xi\right]=\xi\right] \forall \xi\in (-\infty,-1]
\end{equation}
and
\begin{equation}
\label{id2}
W_{-1}(\bar{x})e^{W_{-1}(\bar{x})}=\bar{x} \ \ \left[\Rightarrow f\left(W_{-1}(\bar{x})\right)=\bar{x}\right] \forall \bar{x}\in[-\frac{1}{e},0)
\end{equation}
Also  we denote by $W_{0}$ the restriction to the interval $[-1/e,0)$
of the increasing inverse of $f(\xi)=\xi e^\xi$.
The two identities hold true:
\begin{equation}
\label{id3}
W_{0}(\xi e^\xi)=\xi, \ \ \left[\Leftrightarrow W_0\left[f(\xi)\right]=W_0\left[\xi e^\xi\right]=\xi\right],\ \  \forall \xi\in[-1,0)
\end{equation}
and
\begin{equation}
\label{id4}
W_{0}(\bar{x})e^{W_{0}(\bar{x})}=\bar{x} \ \ \left[\Rightarrow f\left(W_0(\bar{x})\right)=\bar{x}\right]\ \ \forall \bar{x}\in[-\frac{1}{e},0)
\end{equation}
\subsubsection{Numerical values.}
Let us assume that $\bar{x}$ is a solution of our equation.
\begin{equation}
\label{3}
e^{\bar{x}}-2\bar{x}=1
\end{equation}
In order to use the Lambert function $W$, we observe that from (\ref{3}) we get the equivalences
$$e^{\bar{x}}-2\bar{x}=1\Leftrightarrow e^{e^{\bar{x}}-2\bar{x}}=e$$
whence
$-\tfrac12e^{\bar{x}}e^{-\frac12\large e^{\bar{x}}}=-\tfrac12e^{-\frac12}$.
Therefore we can identifies $-\tfrac12e^{\bar{x}}$ with $W\left(-\tfrac12e^{-\frac12}\right)$. Since
$-\frac{1}{e}<-\tfrac12e^{-\frac12}<0,$
the equation which defines the function $W$ of Lambert, has two branches which verifies the same equation
$W(x) e^{W(x)}=-\tfrac12e^{-\frac12}$
and we will have
\begin{equation}
\label{5}
-\tfrac12e^{\bar{x}}=W_0\left(-\tfrac12e^{-\frac12}\right)
\end{equation}
and
\begin{equation}
\label{6}
-\tfrac12e^{\bar{x}}=W_{-1}\left(-\tfrac12e^{-\frac12}\right).
\end{equation}
We call $\bar{x}_1$ the $\bar{x}$ solution of (\ref{5}), and $\bar{x}_2$ the solution of (\ref{6}).

We state easy that $\bar{x}_1=0$. In fact from (\ref{5}) we have
$$-\tfrac12e^{\bar{x}_1}=-\tfrac12 \left[=W_0\left(-\tfrac12e^{-\frac12}\right)\right]$$
and from $e^{\bar{x}_1}=1$, easily follows $\bar{x}_1=0$.

From (\ref{6}), and the relation (\ref{id2}) we get
$e^{\bar{x}_2}=-2W_{-1}\left(-\tfrac12e^{-\frac12}\right),$
and so $\bar{x}_2$:
$$\ln\left(-2W_{-1}\left(-\tfrac12e^{-\frac12}\right)\right)=-\ln\frac{1}{-2W_{-1}\left(-\tfrac12e^{-\frac12}\right)}$$
Now we multiply numerator and denominator by $e^{\frac12}$
$$=-\ln\left[\frac{-\tfrac12 e^{-\frac12}}{W_{-1}\left(-\tfrac12e^{-\frac12}\right)} e^{\frac12}\right] = -\frac{1}{2}-\ln\frac{-\tfrac12e^{-\frac12}}{W_{-1}\left(-\tfrac12e^{-\frac12}\right)}$$
By (\ref{id2}) we have
$$\bar{x}_2=-\frac{1}{2}-W_{-1}\left(-\tfrac12e^{-\frac12}\right).$$
The value $-\frac{1}{2}-W_{-1}\left(-\tfrac12e^{-\frac12}\right)$ is called the parameter of Oseen, or Lamb-Oseen constant, denoted by $\alpha$. Numerical estimates give
$$\alpha=1.25643...$$
We have introduced the Lambert function $W$ in order to give an useful expression to the root of equation
\begin{equation}
\label{walpha}
e^{\alpha}=2\alpha+1.
\end{equation}
In \cite{Ava} we have proved that the real number $\alpha$ is transcendental, through application of the Lindemann - Weierstrass theorem \cite{Gelf}.

\section{Proof of the results}
\label{sec:proofs}
In the following we give the results of the paper.
\subsection{Proof of Proposition \ref{t:alv}}
\begin{proof}[Proof of Proposition \ref{t:alv}]
Write
$$\lambda:=\left\|\sum_n  c_{n}\left(\operatorname{sinc}(n-\xi)- \operatorname{sinc}(\lambda_n-\xi)\right)\right\|^2_{L^2(\R)}.$$
The Fourier transform of the function $ t\to e^{i t\mu}\chi_{[-\pi,\pi]}(t)$ is $\xi\to 2\pi\operatorname{sinc}(\mu-\xi)$. In fact:
$$\F\left(e^{i t\mu}\chi_{[-\pi,\pi]}(t)\right)(\xi)=\int_{-\pi}^\pi e^{i(\mu-\xi)t} dt=2\pi\operatorname{sinc}(\mu-\xi).$$
By Plancherel's theorem
$$\left\|\sum_n  c_{n}\left(\operatorname{sinc}(n-\xi)- \operatorname{sinc}(\lambda_n-\xi)\right)\right\|^2_{L^2(\R)}=
\left\|\sum_n  c_{n} \chi_{[-\pi,\pi]}(t)\left(e^{i n t}-e^{i \lambda_n t}\right)\right\|_{L^2(\R)}^2 $$
$$=\left\|\sum_n  c_{n}  \left(e^{i n t}-e^{i \lambda_n t}\right)\right\|_{L^2(-\pi, \pi)}^2$$
and so, following the proof of Kadec's theorem (see e.g. \cite{You}), when $L<\frac 14$ then $\lambda\leq1-\cos(\pi L)-\sin(\pi L) <1$.
Since $\left\{\operatorname{sinc}(n - \xi)\right\}$ is a Riesz basis of $PW_{\pi}$, the Paley-Wiener criterion shows that also $ \{\operatorname{sinc}(\lambda_n-\xi)\}$ is a Riesz basis of $PW_{\pi}$.

Constant $1/4$ is optimal also for $\left\{\operatorname{sinc}( \lambda_n - \xi)\right\}$. A counterexample due to Ingham \cite{Ing} prove that the set $\{e^{i \lambda_n t}\}$ is not a Riesz basis of $L^2(-\pi,\pi)$ when
\begin{equation}
\lambda_n=
\begin{cases}
 n+\frac{1}{4}, \ \ \ \ \ \ n>0\\
 0, \ \ \ \ \ \ \ \ \ \ \ \   n=0 \\
 n-\frac{1}{4}, \ \ \ \ \ \ n<0 \\
\end{cases}
\end{equation}
Since $PW_\pi$ is isometrically equivalent to $L^2(-\pi,\pi)$ via Fourier transform, the set $\left\{\operatorname{sinc}( \lambda_n - \xi)\right\}$ is not a Riesz basis.
\end{proof}
\begin{corollary}
\label{c:alv}
Let $\{x_n'\}_{n\in\mathbb Z}$ be a system biorthogonal to $\{\operatorname{sinc}(\cdot-\lambda_n) \}_{n\in \mathbb Z}$. Let $\{\lambda_n\}_{n\in\ZZ}$ be a sequence of real numbers for which
\begin{equation}
|\lambda_n-n|\leqq L<\infty, \ \ n=0, \pm 1, \pm 2, ...
\end{equation}
If $L<\frac 14$, and if $f \in PW_\pi$ is continuous, then
\begin{equation}
\label{eq:8c}
f(t)=\sum_{n\in \mathbb Z}\langle f, x_n'\rangle_{PW_\pi}\, \operatorname{sinc}(t-\lambda_n).
\end{equation}
\end{corollary}
\begin{proof}
Let $\{\lambda_n\}_{n\in\ZZ}$ be a sequence of real numbers for which
\begin{equation}
|\lambda_n-n|\leqq L<\frac{1}{4}, \ \ n=0, \pm 1, \pm 2, ...
\end{equation}
We denote with the sequence $\{x_n\}_{n\in\mathbb Z}$ the system $\operatorname{sinc}(t-\lambda_n)$. From Theorem \ref{t:alv}, the sequence $\{x_n\}_{n\in\mathbb Z}$ forms a Riesz basis.

Recall that a sequence $\{x_n\}_{n\in\mathbb Z}$ in an Hilbert space $H$ is a Riesz basis if and only if any element $x\in H$ has a unique expansion $x=\sum_{n\in\mathbb Z} c_n x_n$ with $\{c_n\}_{n\in\mathbb Z}\in \ell^2$. If $\{x_n\}_{n\in\mathbb Z}$ is a Riesz basis, then in the above expansion the Fourier coefficients $c_n$ are given by $c_n=\langle x, x_n'\rangle$, where $\{x_n'\}_{n\in\mathbb Z}$ is a system biorthogonal to $\{x_n\}_{n\in\mathbb Z}$, i.e., a system which satisfies the condition $\langle x_k, x_n'\rangle=\delta_{k,n}$ for all $k$, $n\in\mathbb Z$. 
\end{proof}

\subsection{Proof of Theorem \ref{t:alv11}}
In order to prove Theorem \ref{t:alv11}, we prove the following Lemma.
\begin{lemma}
\label{t:alv1}
Define
$$\mathcal{I}=\left\|\sum_{n} c_n \left[\operatorname{sinc}(\lambda_n-t) - \operatorname{sinc}(n-t)\right]\right\|_{L^2(\mathbb R)}^2.$$
Then
\begin{equation}
\label{eq:5t}
\mathcal{I}\leq2 \sum_n \left[1-\operatorname{sinc}(\lambda_n-n)\right], \ \ n=0, \pm 1, \pm 2, \dots .
\end{equation}
\end{lemma}
\begin{proof}
Write,
\begin{equation}
\label{eq:0b}
\mathcal{I}=\left\|\sum_{n} c_n \left[\operatorname{sinc}(\lambda_n-t) - \operatorname{sinc}(n-t)\right]\right\|_{L^2(\mathbb R)}^2.
\end{equation}
First, we develop function $\operatorname{sinc}(\lambda_n-t)$ respect to basis $\{\operatorname{sinc}(\lambda_n-t)\}_{n\in\mathbb Z}$. We find:
\begin{equation}
\label{eq:0}
\operatorname{sinc}(\lambda_n-t)=\sum_{k\in\mathbb Z} \operatorname{sinc}(\lambda_n-k) \operatorname{sinc}(k-t)
\end{equation}
The convergence in $L^2(\mathbb R)$ is insured by
\begin{equation}
\label{eq:1}
\sum_{k\in\mathbb Z} \operatorname{sinc}^2(\lambda_n-k)= \int_{\mathbb R} \operatorname{sinc}^2(\lambda_n-t) dt=1
\end{equation}
Thanks to equation (\ref{eq:0}) we obtain:
\begin{equation}
\label{eq:2}
\sum_{n} c_n \left[\operatorname{sinc}(\lambda_n-t) - \operatorname{sinc}(n-t)\right]=\sum_{n} c_n \sum_{k\in\mathbb Z} \left[\operatorname{sinc}(\lambda_n-k) - \operatorname{sinc}(n-k)\right] \operatorname{sinc}(k-t).
\end{equation}
This transformation is obvious because
$$\operatorname{sinc}(n-k)=\begin{cases}
0 \ \ \text{for}\ n,k\in\mathbb Z\ \text{and}\ n\neq k \\
1 \ \ \text{for}\ n=k. \\
\end{cases}$$
We obtain, substituting in (\ref{eq:0b}):
\begin{align}
\mathcal{I}&=\left\|\sum_{n} c_n \sum_{k\in\mathbb Z} \left[\operatorname{sinc}(\lambda_n-k) - \operatorname{sinc}(n-k)\right] \operatorname{sinc}(k-t)\right\|_{L^2(\mathbb R)}^2 \notag \\
&=\left\| \sum_{k\in\mathbb Z} \left\{\sum_{n} c_n \left[\operatorname{sinc}(\lambda_n-k) - \operatorname{sinc}(n-k)\right]\right\} \operatorname{sinc}(k-t)\right\|_{L^2(\mathbb R)}^2
\end{align}
Applying the Parseval equality,
$$\mathcal{I}=\sum_{k\in\mathbb Z} \left|\sum_{n} c_n \left[\operatorname{sinc}(\lambda_n-k) - \operatorname{sinc}(n-k)\right] \right|^2$$
Using H\"{o}lder-Schwarz to the sum of products contained in the absolute value, and the condition on $\sum_n |c_n|^2\leq 1$ we have:
\begin{align}
\label{eq:3}
\mathcal{I}&\leq\sum_{k\in\mathbb Z} \sum_{n}\left[\operatorname{sinc}(\lambda_n-k) - \operatorname{sinc}(n-k)\right]^2 = \notag \\
&= \sum_{n} \sum_{k\in\mathbb Z\backslash \{n\}}\left[ \operatorname{sinc}^2(\lambda_n-k) +\left(\operatorname{sinc}(\lambda_n-n) - 1\right)^2\right]
\end{align}
From (\ref{eq:1}),
\begin{equation}
\label{eq:4}
\sum_{k\in\mathbb Z\backslash \{n\}} \operatorname{sinc}^2(\lambda_n-k)=1-\operatorname{sinc}^2(\lambda_n-n)
\end{equation}
Finally, from (\ref{eq:3}) and (\ref{eq:4}), we obtain 
\begin{equation}
\label{eq:5}
\mathcal{I}\leq2 \sum_n \left[1-\operatorname{sinc}(\lambda_n-n)\right].
\end{equation}
\end{proof}

\begin{proof}[Proof of Theorem \ref{t:alv11}]
Let us consider
$$\mathcal{I}\leq2 \sum_n \left[1-\operatorname{sinc}(\lambda_n-n)\right]$$
for $\lambda_n-n=\frac{A}{n^\alpha}$:
$$\mathcal{I}\leq2 \sum_n \left[1-\operatorname{sinc}\left(\frac{A}{n^\alpha}\right)\right]$$
Since, for $x\in(0,\pi/2)$
$$\left|1-\frac{\sin x}{x}\right|=\left|\frac{\sin x}{x}\right| \left|\frac{x}{\sin x} - 1 \right|\leq \frac{1}{\cos x} - 1=\frac{\sin^2 x}{\cos x(1+\cos x)}\leq \frac{\sin^2 x}{\cos x},$$
we have:
$$\mathcal{I}\leq2 \sum_n \left[1-\operatorname{sinc}\left(\frac{A}{n^\alpha}\right)\right]\leq 2 \sum_n \frac{\sin^2\left(\frac{\pi A}{n^\alpha}\right)}{\cos \left(\frac{\pi A}{n^\alpha}\right)}.$$
It is verified if $\frac{\pi A}{n^\alpha}\in\left(0,\frac{\pi}{2}\right)$. Let, for example, $\frac{\pi A}{n^\alpha}\in\left(0,\frac{\pi}{4}\right]$. Hence,
$$\mathcal{I}\leq 2 \sum_n \frac{\sin^2\left(\frac{\pi A}{n^\alpha}\right)}{\cos \left(\frac{\pi A}{n^\alpha}\right)}\leq 2 \sqrt{2} (\pi A)^2 \sum_n \frac{1}{n^{2\alpha}}= 2 \sqrt{2} (\pi A)^2 \zeta(2\alpha).$$
Then
$$\mathcal{I}\leq 2 \sqrt{2} (\pi A)^2 \zeta(2\alpha)<1$$
if $$A<\frac{1}{\pi\sqrt{2 \sqrt{2}\zeta(2\alpha)}}.$$
Moreover, $\zeta(2\alpha)<1$ and, under condition on $A$ and for $\alpha>1/2$, it is confirmed that $\frac{\pi A}{n^\alpha}\in\left(0,\frac{\pi}{4}\right]$.
\end{proof}

\subsection{Proof of Theorem \ref{t:1}.}
\label{sec:complex}

In this Section we study the system $\{\operatorname{sinc}(\lambda_n-t)\}_{n\in\mathbb Z}$ for $\lambda_n\in\mathbb C$ and $|\lambda_n-n|\leq L<\infty$. First, we state the following Lemma whose proof is left to the reader.
\begin{lemma}
\label{l:0}
Let $n$, $k\in\mathbb N$.
The Fourier transform of the function
$ t\to \frac{d^k}{dt^k}\operatorname{sinc}(t-n)$
is $\xi \to (i\xi)^k \chi_{[-\pi,\pi]}(t)\, e^{-in\xi}$.
\end{lemma}
The following is a stability result for $\{\operatorname{sinc}(\lambda_n-t)\}_{n\in\mathbb Z}$ when $\lambda_n\in\mathbb C$. The result involves the Lamb-Oseen constant, as already announced in \cite{Ava}. However the stability bound for the case $\lambda_n\in\mathbb C$ is $\frac{1}{\pi}\, \sqrt\frac{3\alpha}{8}$ (and not $\frac{\alpha}{\pi}$, as written in \cite{Ava}).
\begin{proof} [Proof of Theorem \ref{t:1}]
Let us consider,
\begin{equation}
\label{eq:0bb}
\mathcal{I}=\left\|\sum_{n} c_n \left[\operatorname{sinc}(\lambda_n-t) - \operatorname{sinc}(n-t)\right]\right\|_{L^2(\mathbb R)}^2.
\end{equation}
whenever $\sum_n |c_n|^2\leqq 1$. We use the Taylor series of $\operatorname{sinc}( \lambda_n - t)$:
$$\operatorname{sinc}(n - t)+\sum_{k=1}^{+\infty}\frac{\left(\lambda_n-n\right)^k}{k!} \frac{d^k}{dx^k}\operatorname{sinc}(x-t) \Bigl|_{x=n}$$
Then
\begin{align}
\label{eq:14}
\mathcal{I}&=\left\|\sum_{n} c_n   \sum_{k=1}^{+\infty}\frac{\left(\lambda_n-n\right)^k}{k!} \frac{d^k}{dx^k}\operatorname{sinc}(x-t) \Bigl|_{x=n}  \right\|_{L^2(\mathbb R)}^2 \notag \\
&=\left\| \sum_{k=1}^{+\infty}\frac{1}{k!} \sum_{n} c_n   \left(\lambda_n-n\right)^k \frac{d^k}{dx^k}\operatorname{sinc}(x-t) \Bigl|_{x=n}  \right\|_{L^2(\mathbb R)}^2 \notag \\
&\leq \sum_{k=1}^{+\infty}\frac{1}{k!} \left\|\sum_{n} c_n   \left(\lambda_n-n\right)^k \frac{d^k}{dx^k}\operatorname{sinc}(x-t) \Bigl|_{x=n}  \right\|_{L^2(\mathbb R)}^2
\end{align}
The term $\|\cdot\|$ is reducible to
$$\left\|\sum_{n} c_n   \left(\lambda_n-n\right)^k \frac{d^k}{dx^k}\operatorname{sinc}(x-t) \Bigl|_{x=n}  \right\|_{L^2(\mathbb R)}^2=$$
$$=\sum_{n,m}a_n \overline{a_m}\ \left\langle \frac{d^k}{dx^k}\operatorname{sinc}(x-t) \Bigl|_{x=n}, \frac{d^k}{dx^k}\operatorname{sinc}(x-t) \Bigl|_{x=m} \right\rangle_{L^2(\mathbb R)}$$
where $a_n:=c_n\, \left(\lambda_n-n\right)^k$. Observing that
$$\frac{d^k}{dx^k}\operatorname{sinc}(x-t) \Bigl|_{x=n}=\begin{cases}
- \frac{d^k}{dt^k}\operatorname{sinc}(t-n), \ \ k\ \text{odd} \\
 \frac{d^k}{dt^k}\operatorname{sinc}(t-n),\ \ \ \ k\ \text{even} \\
\end{cases}$$
i.e., $\frac{d^k}{dx^k}\operatorname{sinc}(x-t) \Bigl|_{x=n}=(-1)^k \frac{d^k}{dt^k}\operatorname{sinc}(t-n)$. Then
\begin{align}
\label{eq:12}
&\left\langle \frac{d^k}{dx^k}\operatorname{sinc}(x-t) \Bigl|_{x=n}, \frac{d^k}{dx^k}\operatorname{sinc}(x-t) \Bigl|_{x=m} \right\rangle_{L^2(\mathbb R)}\notag \\
=&\int_{\mathbb R} (-1)^k \frac{d^k}{dt^k}\operatorname{sinc}(t-n) \, \overline{(-1)^k \frac{d^k}{dt^k}\operatorname{sinc}(t-m)}\, dt
\end{align}
From Plancherel's equality and Lemma \ref{l:0}, we have
$$\int_{\mathbb R} (-1)^k \frac{d^k}{dt^k}\operatorname{sinc}(t-n) \, \overline{(-1)^k \frac{d^k}{dt^k}\operatorname{sinc}(t-m)}\, dt=\frac{1}{2\pi}\int_{-\pi}^\pi \xi^{2k} \, e^{i(m-n)\xi}\, d\xi.$$
Hence,
\begin{equation}
\label{eq:13}
\left\langle \frac{d^k}{dx^k}\operatorname{sinc}(x-t) \Bigl|_{x=n}, \frac{d^k}{dx^k}\operatorname{sinc}(x-t) \Bigl|_{x=m} \right\rangle_{L^2(\mathbb R)}= \frac{1}{2\pi}\int_{-\pi}^\pi \xi^{2k} \, e^{i(m-n)\xi}\, d\xi.
\end{equation}
Taking equation (\ref{eq:13}) in (\ref{eq:14}), we obtain:
$$\mathcal{I}\leq \frac{1}{2\pi} \sum_{k=1}^{+\infty}\frac{1}{k!}\sum_{n,m}a_n \overline{a_m}\, \int_{-\pi}^\pi \xi^{2k} \, e^{i(m-n)\xi}\, d\xi:=\omega_1+\omega_2$$
where $\omega_1$, $\omega_2$ are the cases, respectively, when $n=m$ and $n\neq m$. Thereby,
\begin{equation}
\label{eq:15}
\omega_1= \frac{1}{\pi} \sum_{k=1}^{+\infty}\frac{1}{k!}\sum_{n}|a_n|^2 \ \int_{0}^\pi \xi^{2k} \, d\xi\leq \sum_{k=1}^{+\infty}\frac{(\pi L)^{2k}}{k!(2k+1)}.
\end{equation}
and using integration by parts for $\omega_2$, we see that
\begin{align}
\omega_2&= \frac{1}{2\pi} \sum_{k=1}^{+\infty}\frac{1}{k!}\sum_{\substack{n,m \\ n\neq m}} \frac{a_n\, \overline{a_m}}{i(m-n)}\, \int_{-\pi}^\pi \xi^{2k} \, \left[e^{i(m-n)\xi}\right]'\, d\xi \notag \\
&=\frac{1}{2\pi} \sum_{k=1}^{+\infty}\frac{1}{k!}\sum_{\substack{n,m \\ n\neq m}} \frac{a_n\, \overline{a_m}}{i(m-n)}\, \left[\left(\xi^{2k}\, e^{i(m-n)\xi}\right)_{-\pi}^\pi -2k\,\int_{-\pi}^\pi \xi^{2k-1} \, e^{i(m-n)\xi}d\xi\right]\notag \\
&=\frac{i}{\pi} \sum_{k=1}^{+\infty}\frac{1}{(k-1)!}\sum_{\substack{n,m \\ n\neq m}} \frac{a_n\, \overline{a_m}}{m-n}\,\int_{-\pi}^\pi \xi^{2k-1} \, e^{i(m-n)\xi}d\xi.
\end{align}
Putting double series into integral,
$$\omega_2=\frac{i}{\pi} \sum_{k=1}^{+\infty}\frac{1}{(k-1)!}\,\int_{-\pi}^\pi \xi^{2k-1} \, \sum_{\substack{n,m \\ n\neq m}} \frac{a_n e^{-in\xi}\, \overline{a_m e^{-im\xi}}}{m-n}d\xi$$
and denoting $b_n:=a_n e^{-in\xi}$, $\omega_2$ is estimable from above as:
$$|\omega_2|\leq\frac{1}{\pi} \sum_{k=1}^{+\infty}\frac{1}{(k-1)!}\,\int_{-\pi}^\pi |\xi|^{2k-1} \, \left|\sum_{\substack{n,m \\ n\neq m}} \frac{b_n \overline{b_m}}{m-n}\right|d\xi.$$
From Hilbert's inequality for the double series into the integral, we obtain
\begin{equation}
\label{eq:16}
|\omega_2|\leq\sum_{k=1}^{+\infty}\frac{1}{(k-1)!}\,\int_{-\pi}^\pi |\xi|^{2k-1} \, \sum_{n} \left|b_n\right|^2d\xi\leq \sum_{k=1}^{+\infty}\frac{(\pi L)^{2k}}{k!}.
\end{equation}
Then,
$$\mathcal{I}\leq |\omega_1|+|\omega_2|\leq \sum_{k=1}^{+\infty}\frac{(\pi L)^{2k}}{k!(2k+1)}+ \sum_{k=1}^{+\infty}\frac{(\pi L)^{2k}}{k!}=\sum_{k=1}^{+\infty}\frac{(\pi L)^{2k}}{k!}\left[1+\frac{1}{2k+1}\right].$$
We notice that
$$\frac{(\pi L)^{2k}}{k!}\left[1+\frac{1}{2k+1}\right]\leq \frac{(\pi L)^{2k}}{(k+1)!}\, \left(\frac{8}{3}\right)^{k}$$
for all $k\in\mathbb N$ and, in fact,
$$2\, \frac{k+1}{2k+1}\leq \frac{1}{k+1}\, \left(\frac{8}{3}\right)^{k}$$
is verified for all $k\in\mathbb N$. Only for $k=1$ the above inequality becomes an equality. Accordingly,
$$\mathcal{I}\leq \sum_{k=1}^{+\infty} \frac{x^k}{(k+1)!}=\frac{1}{x}\left(e^x-x-1\right), \ \ \text{where}\ x=\frac{8}{3}\, \pi^2 L^2.$$
Set $\lambda=\frac{1}{x}\left(e^x-x-1\right)$ where $x=\frac{8}{3}\, \pi^2 L^2$. In order to get $\lambda< 1$, we solve, in a first moment, the equation $\lambda=1$, that is
\begin{equation}\label{w1}
e^x=2x+1
\end{equation}
We obtain same useful properties on the solutions of equation (\ref{w1}), using the Lambert Function W.
\end{proof}
Numerical estimates give
$$\frac{1}{\pi}\, \sqrt\frac{3\alpha}{8}=0.218492\dots.$$


\section{Tables}
\label{sec:ex}
Let
\begin{equation}
\label{eq:5tt}
\mathcal{I}\leq2 \sum_n \left[1-\operatorname{sinc}(\lambda_n-n)\right]=2 \sum_n\, \sum_{l=1}^\infty (-1)^{l+1}\, \frac{[\pi(\lambda_n-n)]^{2l}}{(2l+1)!}
\end{equation}
In order to numerically valuate $\mathcal{I}$ let, for $n=1,2,3,\dots$,
$$\lambda_n-n=\frac{A}{n^\alpha}, \ \ \ \text{where} \ \alpha>\frac{1}{2}.$$
Substituting in (\ref{eq:5tt}) we obtain
$$\mathcal{I}\leq 2 \sum_n\, \sum_{l=1}^\infty (-1)^{l+1}\, \frac{(\pi A)^{2l}}{(2l+1)!}\, \frac{1}{n^{2l\alpha}}=2\sum_{l=1}^\infty (-1)^{l+1}\, \frac{(\pi A)^{2l}}{(2l+1)!}\, \zeta(2\alpha l)$$
where $\zeta(2l\alpha)$ is the Riemann zeta function. We have the estimate
$$\mathcal{I}\leq 2 \sum_{l=1}^\infty (-1)^{l+1}\, \frac{(\pi A)^{2l}}{(2l+1)!}+ 2 \sum_{l=1}^\infty (-1)^{l+1}\, \frac{(\pi A)^{2l}}{(2l+1)!}\, \left[\zeta(2\alpha l)-1\right]$$
$$=2\left(1-\operatorname{sinc} A\right)+2 \sum_{l=1}^\infty (-1)^{l+1}\, \frac{(\pi A)^{2l}}{(2l+1)!}\, \left[\zeta(2\alpha l)-1\right]$$
In the following we evaluate numerically the expression
$$\lambda:=2\left(1-\frac{\sin \pi A}{\pi A}\right)+2\sum_{l=1}^\infty (-1)^{l+1}\, \frac{\left(\pi A\right)^{2l}}{(2l+1)!}\left[\zeta(2l\alpha)-1\right].$$
Parameters $\alpha$ and $A$ derive from position $\lambda_n-n=\frac{A}{n^\alpha}$, for $\alpha>\frac{1}{2}$ and $A>0$. From Paley-Wiener $\lambda$ must be less than $1$. For a better comprehension we fix:
$$\lambda_1=2\left(1-\frac{\sin \pi A}{\pi A}\right), \ \ \ \lambda_2=2\sum_{l=1}^\infty (-1)^{l+1}\, \frac{\left(\pi A\right)^{2l}}{(2l+1)!}\left[\zeta(2l\alpha)-1\right]$$
Below we try with $A=0.25$ and varying $\alpha$.
$$\ $$
\begin{center}
\begin{tabular}{p{1.5 cm}p{1.5cm}p{1.5cm}p{1.5cm}p{1.5cm}}
$\alpha$&A&$\lambda_1$&$\lambda_2$&$\lambda$\\ \hline
0.7&0.25&0.199367&0.431376&0.630743\\ \hline
0.65&0.25&0.199367&0.600929&0.800296\\ \hline
0.63&0.25&0.199367&0.705618&0.904986\\ \hline
0.62&0.25&0.199367&0.771134&0.970502\\ \hline
0.61599&0.25&0.199367&0.800596&0.999963\\ \hline
\end{tabular}
\label{tab:1}
\end{center}
$$\ $$
Notice that for $\lambda_n-n=\frac{0.25}{n^\alpha}$ ($0.25$ is just the Kadec's bound for exponential bases), when we have, for example, $\alpha=0.7$ (first row of previous table) the parameter $\lambda$ is still far from $1$, which is the maximum value for $\lambda$ in the Paley-Wiener criterion. For decreasing value of $\alpha$, when $\alpha=0.61599$, $\lambda$ is very close to $1$. If $n=1,2,dots$, $\lambda_n-n\leq 0.25$ for all $n\in\mathbb N$.

We now fix $\alpha=1$ while $A$ is variable.
$$\ $$
\begin{center}
\begin{tabular}{p{1.5 cm}p{1.5cm}p{1.5cm}p{1.5cm}p{1.5cm}}
$\alpha$&A&$\lambda_1$&$\lambda_2$&$\lambda$\\ \hline
1&0.25&0.199367&0.132089&0.331456\\ \hline
1&0.35&0.379336&0.257921&0.637257\\ \hline
1&0.4&0.486347&0.336085&0.822432\\ \hline
1&0.42&0.531859&0.370154&0.902013\\ \hline
1&0.44&0.578765&0.405809&0.984574\\ \hline
1&0.44366&0.587491&0.412505&0.999996\\ \hline
\end{tabular}
\label{tab:2}
\end{center}
$$\ $$

At this point, we have $\lambda_n-n=\frac{A}{n}$. When $A=0.25$ (first row of previous table) the parameter $\lambda$ is still far from value $1$ of $\lambda$ in the Paley-Wiener criterion. For increasing value of $A$, when $A\simeq 0.443\dots$, $\lambda$ is very close to $1$. If $n=1,2,\dots$, $\lambda_n-n\leq L$ where $L$ seems to be approximately $0.443\dots$, which is greater of Kadec's bound. We have completed here the study announced in \cite{Ava}, giving a whole proof for stability of sinc bases.

\begin{figure}[tb]
\centering
\includegraphics[scale=0.80]{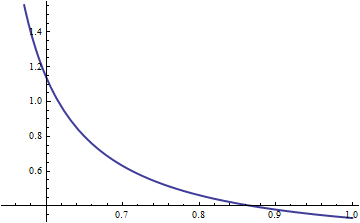}
\caption{Plot of $2\left(1-\frac{\sin \pi A}{\pi A}\right)+2\sum_{l=1}^\infty (-1)^{l+1}\, \frac{\left(\pi A\right)^{2l}}{(2l+1)!}\left[\zeta(2l\alpha)-1\right]$. Here $A=0.25$, horizontal axis is referred to $\alpha$ and the graphics is obtained in the range $\alpha\in[0.55,1]$. }
\label{fig:L_4}
\end{figure}

\end{document}